\documentclass[12pt]{amsart}

\usepackage[margin=1in]{geometry}

\usepackage{amsmath,amssymb,amsthm}
\usepackage{xcolor}

\newtheorem{theorem}{Theorem}
\newtheorem{lemma}[theorem]{Lemma}
\theoremstyle{remark} \newtheorem*{remark}{Remark}
\theoremstyle{remark} \newtheorem*{example}{Example}
\renewcommand{\a}{\mathbf{a}}

\numberwithin{equation}{section}
\numberwithin{theorem}{section}

\newcommand{\binomial}[2]{\left(#1 \atop #2\right)}

\title{Upper bounds on number fields of given degree and bounded discriminant}
\author{Robert J. Lemke Oliver}
\author{Frank Thorne}

\begin{document}

\begin{abstract}
Let $N_n(X)$ denote the number of degree $n$ number fields with discriminant bounded by $X$.  In this note, we improve the best known upper bounds on $N_n(X)$, finding that $N_n(X) = O(X^{ c (\log n)^2})$ for an explicit constant $c$.
\end{abstract}

\maketitle

\section{Introduction and statement of results}

Let $N_n(X) := \#\{ K/\mathbb{Q} : [K:\mathbb{Q}] = n, |\mathrm{Disc}(K)| \leq X\}$ be the number of degree $n$ extensions of $\mathbb{Q}$ with bounded absolute discriminant $\mathrm{Disc}(K)$.  It follows from the Hermite--Minkowski theorem that $N_n(X)$ is finite, and in fact bounded by $O_n(X^n)$.  This was substantially improved by Schmidt \cite{Schmidt}, who shows that $N_n(X) \ll X^{(n+2)/4}$, by Ellenberg and Venkatesh \cite{EllenbergVenkatesh}, who obtain an exponent that is $O(\exp(c\sqrt{\log n}))$ for some constant $c$, and Couveignes \cite{Couveignes}, who shows that $N_n(X) \ll X^{c (\log n)^3}$ for some unspecified constant $c$.

Our main theorem improves on these results.

\begin{theorem}\label{thm:main-intro}
There is a constant $c>0$ such that $N_n(X) \ll_n X^{c (\log n)^2}$ for every $n \geq 6$.  Explicitly, we may take $c=1.564$, and for every $c^\prime > 1/(4(\log 2)^2) \approx 0.52$, there is some $N>0$ such that $N_n(X) \ll_n X^{c^\prime (\log n)^2}$ for every $n \geq N$.
\end{theorem}

The proof of Theorem \ref{thm:main-intro} follows the same general strategy employed by Ellenberg and Venkatesh and by Couveignes; see Section \ref{sec:motivation} for a loose discussion of the differences.  

In fact, Theorem \ref{thm:main-intro} follows straightforwardly from a numerical computation and by combining the Schmidt bound with the following theorem that is somewhat more flexible than Theorem \ref{thm:main-intro}.

\begin{theorem}\label{thm:intro-summary}  Let $n \geq 2$.

1) Let $d$ be the least integer for which $\left({d+2}\atop{2}\right) \geq 2n+1$.  Then
\begin{equation}\label{eq:main2}
N_n(X) \ll_n X^{2d - \frac{d(d-1)(d+4)}{6n}} \ll X^{\frac{8\sqrt{n}}{3}}.
\end{equation}

2) Let $3 \leq r \leq n$ and let $d$ be such that $\binomial{d+r-1}{r-1} > rn$.  Then $N_n(X) \ll_{n,r,d} X^{dr}$.
\end{theorem}

Optimizing the choice of $d$, the second case of Theorem \ref{thm:intro-summary} yields an exponent that is $O(r^2 n^{1/(r-1)})$ with an absolute implied constant.  We note that, in applying Theorem \ref{thm:intro-summary} to deduce Theorem \ref{thm:main-intro}, we will choose $r$ to be a suitable multiple of $\log n$.

This improves upon the Schmidt bound $N_n(X) \ll X^{\frac{n + 2}{4}}$ for $n \geq 95$. 
For $n \leq 5$, asymptotic formulas of the form $N_n(X) \sim c_n X$ were proved by Davenport-Heilbronn \cite{DH}, Cohen-Diaz y Diaz-Olivier \cite{CDyDO}, and Bhargava \cite{B4,B5}; although our method still applies in these cases, it yields a substantially weaker result.
In general, for $6 \leq n \leq 94$, the Schmidt bound remains the best known, but improvements are available for fields with restricted Galois structure due to work of Dummit \cite{Dummit}.

\section{Setup for the proof}
\label{sec:motivation}

We begin by recalling the central idea of Schmidt's proof in language that will be of use to us.  Let $K$ be a number field of degree $n$.  The ring of integers $\mathcal{O}_K$ is a lattice inside Minkowski space $K_\infty := K \otimes \mathbb{R} \simeq \mathbb{R}^n$ with covolume $c_K \sqrt{\mathrm{Disc}(K)}$, where $c_K$ is a constant depending only on the signature of $K$.  Similarly, the set $\mathcal{O}_K^0$ of integers in $\mathcal{O}_K$ with trace $0$ forms a lattice inside the trace $0$ subspace of $K_\infty$ with covolume $c_K^\prime \sqrt{\mathrm{Disc}(K)}$ for some $c_K^\prime$, and it follows that there is some $\alpha \in \mathcal{O}_K^0$ all of whose embeddings are at most $O(\mathrm{Disc}(K)^{\frac{1}{2n-2}})$.  We assume for convenience of exposition that $\mathbb{Q}(\alpha) = K$; if not, Schmidt proceeds by induction, counting both the possible extensions $F=\mathbb{Q}(\alpha)/\mathbb{Q}$ and the possible $K/F$.  The minimal polynomial $p_\alpha(x)$ of $\alpha$ is given by
\[
p_\alpha(x) = \prod_{\sigma} (x-\sigma(\alpha)) = x^n + a_2(\alpha) x^{n-2} + \dots + a_n(\alpha),
\]
where the coefficients $a_i$ are integers satisfying $|a_i(\alpha)| \ll \mathrm{Disc}(K)^{ \frac{i}{2n-2}}$ and the product runs over the $n$ embeddings $\sigma\colon K \hookrightarrow \mathbb{C}$.  It follows that, ignoring the issue of subfields, the number of degree $n$ fields with discriminant at most $X$ may be bounded by the number of integral polynomials $f(x) = x^n + a_2x^{n-2} + \dots + a_n$ where each $a_i$ is bounded by $O(X^{\frac{i}{2n-2}})$.  The number of such polynomials is $O(X^{\frac{n+2}{4}})$, and this is Schmidt's bound.

The key idea of Ellenberg and Venkatesh's improvement, on which our work as well as Couveignes's is based, is that there are more invariants of small height attached to tuples of integers inside $\mathcal{O}_K^0$.  For example, suppose $\alpha$ and $\beta$ are elements of $\mathcal{O}_K^0$ whose maximum embedding is bounded by some $Y$.  Then $\alpha$ and $\beta$ (and hence $K$) are determined by their minimal polynomials, which in turn are determined by the traces $\mathrm{Tr}(\alpha^i)$ and $\mathrm{Tr}(\beta^i)$ for $1 \leq i \leq n$.  These traces are integers of size $O(Y^i)$, and it follows that there are most $O(Y^{2\sum_{i=2}^n i}) = O(Y^{n^2+n-2})$ possible pairs $(\alpha,\beta)$.  It is possible to do much better by exploiting mixed traces $\mathrm{Tr}(\alpha^i\beta^j)$, however.  

By regarding the $n$ embeddings $\alpha^{(1)}, \dots, \alpha^{(n)}$ and $\beta^{(1)},\dots,\beta^{(n)}$ of $\alpha$ and $\beta$ as variables, we might hope that once $2n$ of these mixed traces are specified, it is possible to recover the values $\alpha^{(1)}, \dots, \alpha^{(n)}, \beta^{(1)},\dots,\beta^{(n)}$, and hence the pair $(\alpha,\beta)$ and the field $K$.  There are $\binomial{d+2}{2}-1$ mixed traces with $i + j \leq d$, so in particular we might hope that the traces $\mathrm{Tr}(\alpha^i\beta^j)$ with $i+j \leq d \approx 2\sqrt{n}$ suffice to determine $\alpha$ and $\beta$.  Since $\mathrm{Tr}(\alpha^i\beta^j) \ll Y^{i+j}$, this would yield that there are at most $Y^{O(n^{3/2})}$ such pairs.

In Lemma \ref{lem:r=2}, we prove that the set of $2n$ traces $\mathrm{Tr}(\alpha^i\beta^j)$ with smallest possible $i+j$ suffices to determine the pair for a ``generic'' choice of $\alpha$ and $\beta$.  We show that every field has such a choice of $\alpha$ and $\beta$ with small height, and this leads to the first case of Theorem \ref{thm:intro-summary}.  This also provides the first insight into our improvement over the work of Ellenberg and Venkatesh, who only prove in this context that a set of roughly $8n$ traces suffices.

Consider now an $r$-tuple $\alpha_1,\dots,\alpha_r \in \mathcal{O}_K$, with $r \geq 3$.  By a similar heuristic as above, we might hope that once $rn$ different traces $\mathrm{Tr}(\alpha_1^{i_1}\dots \alpha_r^{i_r})$ are specified, the tuple $\alpha_1,\dots,\alpha_r$ is determined.  We show that this is the case in Lemma \ref{lem:general-r} for a set of $rn$ traces with $i_1 + \dots + i_r$ nearly as small as possible.  By contrast, Ellenberg and Venkatesh only show that a set of roughly $2^{2r-1}n$ traces suffices.  Couveignes's work is morally similar but structurally a little different; 
instead of working with mixed traces, he constructs a set of $r$ polynomials, each with about $rn$ coefficients, that determines each number field.  Thus, his approach relies on taking roughly $r^2n$ invariants of a number field.

By expressing the trace as a sum over embeddings, we may regard the mixed traces $\mathrm{Tr}(\alpha_1^{i_1}\dots \alpha_r^{i_r})$ as being governed by polynomial maps from $(\mathbb{A}^n)^r$ to $\mathbb{A}^1$.  The key lemma we use to determine the $\alpha_i$ from these traces is then the following.

\begin{lemma}\label{lem:tangent-space}
For $N \geq 1$, let $f_1, \dots, f_N$ be polynomials from $\mathbb{A}^N$ to $\mathbb{A}^1$.  Suppose that the determinant of the matrix $(\frac{\partial f_i}{\partial x_j})_{1 \leq i,j\leq N}$ is not identically $0$.  Then there is a nonzero polynomial $P\colon \mathbb{A}^N \to \mathbb{A}^1$ such that whenever $P(\mathbf{x}_0) \neq 0$ for some $\mathbf{x}_0 \in \mathbb{A}^N$, the variety $V:=V_{\mathbf{x}_0}$ cut out by the equations $f_1(\mathbf{x}) = f_1(\mathbf{x}_0), \dots, f_N(\mathbf{x}) = f_N(\mathbf{x}_0)$ consists of at most $\prod_i (\deg f_i)$ points.

\end{lemma}
\begin{proof}
Let $F\colon \mathbb{A}^N \to \mathbb{A}^N$ be defined by $F(\mathbf{x}) = (f_1(\mathbf{x}),\dots,f_N(\mathbf{x}))$, 
so that $V_{\mathbf{x}_0} = F^{-1}(F(\mathbf{x}_0))$.  $F$ is dominant, since for example the image of a Euclidean neighborhood of any $\mathbf{x}$ for which $\det \big( \frac{\partial f_i}{\partial x_j} \big)(\mathbf{x}) \neq 0$ is a neighborhood of $F(\mathbf{x})$. 
By \cite[Chapter 1, \S 8, Theorem 3]{redbook}, there is a Zariski open set $U \subseteq \textnormal{im}(F)$ such that $F^{-1}(\mathbf{y})$ is of dimension zero for $\mathbf{y} \in U$. Choosing a
polynomial $Q$ vanishing on the complement of $U$, $P = Q \circ F$ is our desired polynomial.

Therefore, when  $P(\mathbf{x}_0) \neq 0$, the variety $V:=V_{\mathbf{x}_0}$ has dimension $0$ and consists of a finite number of points.  The quantitative bound follows by using B\'ezout's theorem \cite[Theorem I.7.7]{har} to iteratively bound the number of affine components of the projectivization of $V\big(f_1 - f_1(\mathbf{x}_0), \dots,  f_m - f_m(\mathbf{x}_0)\big)$ for each $m \leq N$.
\end{proof}



\section{The dimension of mixed trace varieties}

In this section, we show that the varieties associated to fixed values of the ``mixed traces'' introduced in the previous section typically have dimension $0$.  We do so over $\mathbb{C}$.  Thus, let $n \geq 2$ be an integer, corresponding to the degree of the extensions we wish to count, and let $r \geq 2$ be an integer corresponding to the number of elements of which we wish to take the mixed trace.  To an $r$-tuple $\mathbf{a}=(a_1,\dots,a_r) \in \mathbb{Z}_{\geq 0}^r$, we associate the function
\[
\mathrm{Tr}_{n,\mathbf{a}} \colon (\mathbb{A}^n)^r \to \mathbb{A}^1
\] 
given by
\[
\mathrm{Tr}_{n,\mathbf{a}}(\mathbf{x}_1,\dots,\mathbf{x}_r)
	:= \sum_{i=1}^n x_{1,i}^{a_1} \dots x_{r,i}^{a_r}.
\]
We let $|\mathbf{a}| = a_1 + \dots + a_r$ denote the total degree of $\mathrm{Tr}_{n,\mathbf{a}}$.  Motivated by Lemma \ref{lem:tangent-space}, let 
\[
D \mathrm{Tr}_{n,\mathbf{a}}
	:= \Big( \frac{\partial}{\partial x_{k,i}} \mathrm{Tr}_{n,\mathbf{a}} \Big)_{\substack{1 \leq k \leq r \\ 1 \leq i \leq n}}
\]
denote the (row) vector of partial derivatives of $\mathrm{Tr}_{n,\mathbf{a}}$ (i.e. its gradient).  Our goal, then, is to find a set $A$ of $rn$ different vectors $\mathbf{a}$ with small combined total degree for which the determinant of the matrix 
\[
	(D \mathrm{Tr}_{n,\mathbf{a}})_{\mathbf{a} \in A}
\]
is not identically $0$.  We begin by considering the case $r=2$, both to clarify ideas and because we obtain an essentially optimal result in this case.

\begin{lemma}\label{lem:r=2}
Let $n \geq 1$, and let $A_n = (\a_1, \cdots, \a_{2n})$ consist of the first $2n$ elements of the ordered set $\{(1,0),(0,1),(2,0),(1,1),(0,2),\dots\}$, that is, the set of ordered pairs $(i,j)$ ordered first by total degree $i+j$, then by $j$. 

 Then with notation as above, $\mathrm{det} (D\mathrm{Tr}_{n,\mathbf{a}})_{\mathbf{a} \in A_n} \neq 0$.
\end{lemma}
\begin{proof}
Induction on $n$, with $\mathrm{det} (D\mathrm{Tr}_{1,\mathbf{a}})_{\mathbf{a} \in A_1} = 1$.

Every entry in the matrix $\mathbf{D}=(D \mathrm{Tr}_{n,\mathbf{a}})_{\mathbf{a} \in A}$ is a monomial;
the $(k, n)$-entry $\mathbf{D}_{k, n}$ equals $a_{k, 1} x_{1, n}^{a_{k, 1} - 1} x_{2, n}^{a_{k, 2}}$, and
$\mathbf{D}_{k, 2n} = a_{k, 2} x_{1, n}^{a_{k, 1}} x_{2, n}^{a_{k, 2} - 1}.$
We write $\det(\mathbf{D})$ an alternating sum of products of these monomials. For each $k$ and $\ell$,
the contribution of those terms involving either $\mathbf{D}_{k, n}$ and $\mathbf{D}_{\ell, 2n}$, or alternatively 
$\mathbf{D}_{k, 2n}$ and $\mathbf{D}_{\ell, n}$, is
\begin{align}\label{eq:ind_matrix}
& \pm \det \begin{bmatrix} a_{k, 1} x_{1, n}^{a_{k, 1} - 1} x_{2, n}^{a_{k, 2}} & a_{k, 2} x_{1, n}^{a_{k, 1}} x_{2, n}^{a_{k, 2} - 1}  \\ 
a_{\ell, 1} x_{1, n}^{a_{\ell, 1} - 1} x_{2, n}^{a_{\ell, 2}} & a_{\ell, 2} x_{1, n}^{a_{\ell, 1}} x_{2, n}^{a_{\ell, 2} - 1}  \end{bmatrix} \cdot \delta_{k, \ell} \\
= & \label{eq:ind_matrix_2}
\pm \det \begin{bmatrix} a_{k,1} & a_{k, 2} \\ a_{\ell, 1} & a_{\ell, 2} \end{bmatrix} \cdot x_{1, n}^{a_{k, 1} + a_{\ell, 1} - 1} x_{2, n}^{a_{k, 2} + a_{\ell, 2} - 1} 
\cdot \delta_{k, \ell},
\end{align}
where $\delta_{k, \ell}$ is the relevant $(2n - 2) \times (2n - 2)$ matrix minor, which doesn't involve $x_{1,n}$ or $x_{2, n}$. 

By construction, if $\mathbf{a}_{k} + \mathbf{a}_\ell$ = $\mathbf{a}_{2n - 1} + \mathbf{a}_{2n}$ in $\mathbb{Z}^2$, then 
$\{ k, \ell \} = \{ 2n - 1, 2n \}$. Since the exponents of $x_{1, n}$ and $x_{2, n}$ in \eqref{eq:ind_matrix_2} are given by
 $\mathbf{a}_{k} + \mathbf{a}_{\ell} - (1, 1)$, this implies that the contribution \eqref{eq:ind_matrix_2} from
 $(l, \ell) = (2n - 1, 2n)$ is not cancelled by any other contribution. It therefore suffices to prove that this contribution is not zero.
 
But this is immediate: the $2 \times 2$ determinant in \eqref{eq:ind_matrix_2} is nonzero because consecutive elements of $A_n$ are
never scalar multiples of one another, and $\delta_{2n-1, 2n} = \mathrm{det} (D\mathrm{Tr}_{n-1,\mathbf{a}})_{\mathbf{a} \in A_{n - 1}}$.
\end{proof}

\begin{example}
Let $n=3$.  Then $A = \{(1,0),(0,1),(2,0),(1,1),(0,2),(3,0)\}$. The associated matrix is
\[
\mathbf{D} = 
\begin{bmatrix}
\text{\fbox{$1$}} & \text{\fbox{$1$}} & 1 & \text{\fbox{$0$}} & \text{\fbox{$0$}} & 0 \\
\text{\fbox{$0$}} & \text{\fbox{$0$}} & 0 & \text{\fbox{$1$}} & \text{\fbox{$1$}} & 1 \\
\text{\fbox{$2x_{1,1}$}} & \text{\fbox{$2x_{1,2}$}} & 2x_{1,3} & \text{\fbox{$0$}} & \text{\fbox{$0$}} & 0 \\
\text{\fbox{$x_{2,1}$}} & \text{\fbox{$x_{2,2}$}} & x_{2,3} &  \text{\fbox{$x_{1,1}$}} & \text{\fbox{$x_{1,2}$}} & x_{1,3} \\
0 & 0 & \text{\fbox{$0$}} & 2x_{2,1} & 2x_{2,2} & \text{\fbox{$2x_{2,3}$}} \\
3x_{1,1}^2 & 3x_{1,2}^2 & \text{\fbox{$3x_{1,3}^2$}} & 0 & 0 & \text{\fbox{$0$}} \\
\end{bmatrix}.
\]
The boxed entries comprise the $4 \times 4$ and $2 \times 2$ matrices in \eqref{eq:ind_matrix_2} for $(k, \ell) = (2n - 1, 2n)$, whose determinants
are nonzero and multiply to 
a summand of $\det(\mathbf{D})$.
\end{example}

For larger $r$, we apply a theorem due to Alexander and Hirschowitz \cite{AH}, that we state in the following manner to be consistent with our notation.  (See also \cite{BO}.)  This theorem is also an important ingredient in Couveignes's work.

\begin{theorem}[Alexander--Hirschowitz] \label{thm:AH}
Let $V$ denote the complex vector space of homogeneous degree $d$ polynomials in $r$ variables.  Given $n$ general points in $\mathbb{P}^{r-1}$, let $W \subseteq V$ denote the subspace of polynomials whose first order partial derivatives all vanish at each of the $n$ points.  Then $W$ has the ``expected'' codimension in $V$, namely
\[
\mathrm{codim}\,W
	= \min\{ \mathrm{dim}\,V, rn \},
\]
except for the following cases:
\begin{itemize}
\item $d=2$, $2 \leq n \leq r-1$;
\item $d=3$, $r=5$, $n=7$;
\item $d=4$, $(r,n) \in \{ (3,5),(4,9),(5,14) \}.$
\end{itemize}
\end{theorem}

With Theorem \ref{thm:AH}, we are able to find a good choice of the set $A$ in general.

\begin{lemma}\label{lem:general-r}
Let $n \geq 6$, $3 \leq r \leq n$, and suppose $d$ is such that $\left( {d+r-1}\atop{r-1}\right) > rn$.  Then there is a set $A$ of $rn$ vectors $\mathbf{a} \in \mathbb{Z}_{\geq 0}^{r}$ of total degree $d$ for which the determinant $\mathrm{det}(D\mathrm{Tr}_{n,\mathbf{a}})_{\mathbf{a} \in A} \neq 0$.  
\end{lemma}
\begin{proof}
As in Theorem \ref{thm:AH}, let $V$ denote the vector space of complex homogeneous polynomials in $r$ variables with degree $d$.
Choose an arbitrary set of $n$ general points in $\mathbb{P}^{r-1}$, and let $W$ denote the subspace with vanishing first order partials at each.
 Under the hypotheses of Lemma \ref{lem:general-r}, $W$ has codimension $rn$, since $\mathrm{dim}\,V = \left( {d+r-1} \atop {r-1}\right)$ and none of the exceptional cases of Theorem \ref{thm:AH} apply.

Let $\mathfrak{M}_d$ denote the set of monomials of degree $d$, which naturally forms a basis for $V$.  For each $m \in \mathfrak{M}_d$, form an $rn$-dimensional column vector $v_m$ by evaluating each of the first order partials of $m$ at the $n$ general points.  Let $M$ denote the $rn \times \left( {d+r-1} \atop {r-1}\right)$ matrix whose columns are the vectors $v_m$ for $m \in \mathfrak{M}_d$.  
The subspace $W$ may be identified with the kernel of $M$, and since $W$ has the expected codimension, it follows that $M$ has full rank, namely $\mathrm{rank}(M) = rn$.  

Choose an $rn\times rn$ minor $M^{\prime}$ of $M$ of full rank. The columns of $M^{\prime}$ are indexed by monomials of total degree $d$ that may be identified with elements of $\mathbb{Z}_{\geq 0}^r$.  Let $A$ consist of those associated elements in $\mathbb{Z}_{\geq 0}^r$.  Then $M^{\prime}$ is the transpose of the matrix $(D \mathrm{Tr}_{n,\mathbf{a}})_{\mathbf{a} \in A}$, 
evaluated at our set of $n$ points, and since $\mathrm{det}(M^{\prime}) \neq 0$ we have
$\mathrm{det} (D \mathrm{Tr}_{n,\mathbf{a}})_{\mathbf{a} \in A} \neq 0$ as well.
\end{proof}

\begin{remark} In Lemma \ref{lem:general-r} (and hence also Theorem
\ref{thm:intro-summary}) we may also allow $\left( {d+r-1}\atop{r-1}\right) = rn$, provided that $(d, r, n)$ is not $(3, 5, 7)$
or $(4, 5, 14)$.
\end{remark}

\section{Bounds on the number of number fields}

Let $K/\mathbb{Q}$ be a number field of degree $n$.  Then $K$ may be embedded into $\mathbb{C}^n$.  Lemmas \ref{lem:r=2} and \ref{lem:general-r} produce, for any $r \geq 2$, 
a set $A \in \mathbb{Z}_{\geq 0}^r$ for which $\mathrm{det}(D\mathrm{Tr}_{n,\mathbf{a}})_{\mathbf{a} \in A} \neq 0$.  By Lemma \ref{lem:tangent-space}, there is a hypersurface outside of which the variety cut out by specifying the mixed traces $\mathrm{Tr}_{n,\mathbf{a}}$ for $\mathbf{a} \in A$ consists of a bounded number of points.  The following lemma, due to Ellenberg and Venkatesh \cite[Lemma 2.4]{EllenbergVenkatesh}, will be used to show that there is an $r$-tuple of integers in $K$, at least one of which cuts out $K$, of small height that avoids this hypersurface.

\begin{lemma}[Ellenberg--Venkatesh]\label{lem:hypersurface}
Let $P\colon \mathbb{A}^N \to \mathbb{A}^1$ be a polynomial of degree $d$.  Then there are integers $a_1,\dots,a_N$ with $|a_i| \leq (d+1)/2$ for which $P(a_1,\dots,a_N) \neq 0$.
\end{lemma}

Lastly, we shall make use of the following upper bound on the height of the largest Minkowski minimum of a number field that follows from work of Bhargava, Shankar, Taniguchi, Thorne, Tsimerman, Zhao \cite{BSTTTZ}.

\begin{lemma}\label{lem:small-basis}
Given a number field $K$ of degree $n$, there is an integral basis $\{\beta_1,\dots,\beta_n\}$ of its ring of integers for which $|\beta_i| \ll_n \mathrm{Disc}(K)^{1/n}$ in each archimedean embedding.
\end{lemma}

\begin{proof}[Proof of Theorem \ref{thm:intro-summary}]
Consider first the case $r \geq 3$ and let $d$ be as in the statement of the theorem.  Then by Lemma \ref{lem:general-r}, there is a set $A$ of $\mathbf{a} \in \mathbb{Z}_{\geq 0}^r$ of size $rn$ and degree $d$
for which $\mathrm{det}(D\mathrm{Tr}_{n,\mathbf{a}})_{\mathbf{a} \in A} \neq 0$.  By Lemma \ref{lem:tangent-space}, there is a polynomial $P\colon (\mathbb{A}^n)^r \to \mathbb{A}^1$ such that whenever $P(\mathbf{x}_0) \neq 0$, the variety
\[
\{ \mathbf{x} \in (\mathbb{A}^n)^r : \mathrm{Tr}_{n,\mathbf{a}}(\mathbf{x}) = \mathrm{Tr}_{n,\mathbf{a}}(\mathbf{x}_0) \text{ for all $\mathbf{a} \in A$}\}
\]
consists of $O_{d,r,n}(1)$ points.  Given $\mathbf{x} \in (\mathbb{A}^n)^r$, we also define $\mathrm{Disc}^{(1)}(\mathbf{x})$ to denote its discriminant in the first copy of $\mathbb{A}^n$, i.e.
\[
\mathrm{Disc}^{(1)}(\mathbf{x})
	:= \prod_{\substack{i,j \leq n \\ i \neq j}} (x_{1,i}-x_{1,j}).
\]

Now, let $K$ be a number field of degree $n$ and discriminant at most $X$.  By fixing an embedding $K \hookrightarrow \mathbb{C}^n$, we may regard an $r$-tuple of integers $\alpha_1,\dots,\alpha_r \in \mathcal{O}_K$ as giving rise to a point $\mathbf{x}_\alpha \in (\mathbb{A}^n)^r$.  Combining Lemmas \ref{lem:hypersurface} and \ref{lem:small-basis}, we find there are $\alpha_1,\dots,\alpha_r \in \mathcal{O}_K$ with each $|\alpha_i| \ll_{n,d,r} X^{1/n}$ in each archimedean embedding for which the associated point $\mathbf{x}_\alpha$ satisfies both $P(\mathbf{x}_\alpha) \neq 0$ and $\mathrm{Disc}^{(1)}(\mathbf{x}_\alpha) \neq 0$.  Since $\mathrm{Disc}^{(1)}(\mathbf{x}_\alpha) \neq 0$, $\alpha_1$ cuts out a degree $n$ extension of $\mathbb{Q}$, which must therefore be equal to $K$.  
It follows that any such $K$ is determined, up to $O_{n,r,d}(1)$ choices, by the $rn$ values $\mathrm{Tr}_{n,\mathbf{a}}(\mathbf{x}_\alpha)$ for $\mathbf{a} \in A$.  Each of these quantities is an integer of size $O(X^{d/n})$, whence there are $O(X^{rd})$ choices in total, and hence $O(X^{rd})$ number fields $K$. 

The case $r=2$ is similar, except appealing to Lemma \ref{lem:r=2} for the construction of the set $A$, and noting that $\mathrm{Tr}_{n,\mathbf{a}}(\mathbf{x}_\alpha)$ is an integer of size $O(X^{|\mathbf{a}|/n})$.  For the set $A$ produced by Lemma \ref{lem:r=2}, we find
\[
\sum_{\mathbf{a} \in A} |\mathbf{a}|
	= 2nd - \frac{d(d-1)(d+4)}{6},
\]
where $d$ is the least integer for which $\left({d+2} \atop {2}\right) \geq 2n+1$.  This yields the remaining case of the theorem;
for the second inequality in \eqref{eq:main2}, take $d = \lfloor 2 \sqrt{n} \rfloor$.
\end{proof}

\begin{proof}[Proof of Theorem \ref{thm:main-intro}]
We describe an asymptotically optimal choice of $r$ and $d$ as $n\to\infty$ in the second part of Theorem \ref{thm:intro-summary}.  This will show that the exponent in Theorem \ref{thm:main-intro} may be taken to be $(1/(4(\log 2)^2) + o(1)) (\log n)^2$ as $n\to\infty$.

Thus, let $n$ be large.  We choose $d = \alpha \log n$ and $r-1 = \beta \log n$ for constants $\alpha$ and $\beta$.  So doing, a computation with Stirling's formula reveals
\begin{equation}\label{eq:stirling}
\log \left({d+r-1} \atop {r-1}\right)
	= ((\alpha + \beta) \log(\alpha + \beta) - \alpha \log \alpha - \beta \log \beta) \log n + O(\log\log n).
\end{equation}
On the other hand, $\log (rn) = \log n + O(\log\log n)$, so we find that asymptotically optimal choices of $\alpha$ and $\beta$ will satisfy
\[
(\alpha + \beta) \log(\alpha + \beta) - \alpha \log \alpha - \beta \log \beta
	= 1 + O\left(\frac{\log \log n}{\log n}\right).
\]
This expression is symmetric in $\alpha$ and $\beta$, as is the exponent $dr = \alpha\beta (\log n)^2 + O(\log n)$ produced by Theorem \ref{thm:intro-summary}.  As a Lagrange multipliers computation 
shows, the exponent is minimized by choosing $\alpha = \beta = 1/2\log 2$.  This yields the second part of the theorem.

The first part of the theorem (with $c = 1.564$) follows for $n < e^{12}$ by a numerical computation, with $n = 805$ being the bottleneck. For $n > e^{12}$ we choose
$d = r - 1 = \lceil \log(n) \rceil$, apply the bound $\left(2d \atop d\right) \geq \frac{4^d}{2 \sqrt{d}}$, and verify that $\left(2d \atop d\right) \geq rn$ and
$dr < 1.564(\log n)^2$ with this choice.
\end{proof}



\section{Scope for improvement}

We list here three possible directions in which our results may be improved.

First, for $r \geq 3$, it would be desirable to incorporate mixed traces of total degree \emph{at most} $d$, as opposed to restricting attention to mixed traces of total degree exactly $d$ as is done in Lemma \ref{lem:general-r}.  An optimistic version of this improvement would imply the bound $N_n(X) \ll X^{dr}$, say, whenever $\left({d+r} \atop {r}\right) \geq rn$.  This would substantially improve the efficiency of the method for fixed $r$, as this exponent would be $O(r^2 n^{1/r})$, as opposed to that obtained from Theorem \ref{thm:intro-summary}, which is $O(r^2 n^{1/(r-1)})$.  It would not, however, yield an improved version of Theorem \ref{thm:main-intro}.

Second, it would be worthwhile to incorporate greater input from the geometry-of-numbers.  For $r < n$,
there are linearly independent $r$-tuples of integers $\alpha_1,\dots,\alpha_r \in \mathcal{O}_K$ for which $|\alpha_i| \ll \mathrm{Disc}(K)^{\frac{1}{2(n-r+1)}}$, which
for $r \asymp \log(n)$ is significantly smaller than 
the bound $\mathrm{Disc}(K)^{1/n}$ coming from Lemma \ref{lem:small-basis}.  However, in applying Lemma \ref{lem:hypersurface}, it is a priori necessary to work with the full ring of integers, and not a small rank sublattice, in which case Lemma \ref{lem:small-basis} is essentially optimal.  It would be interesting to exclude the possibility that the hypersurface outside of which a mixed trace variety is a complete intersection contains these small rank sublattices for varying $K$.  This would lead to an improvement in the exponents in Theorems \ref{thm:main-intro} and \ref{thm:intro-summary} essentially by a factor of $2$.  The authors hope to return to this question in future work.

Lastly, unlike in the case of Schmidt's work, which considers minimal polynomials, it is not typically the case that algebraic points on a mixed trace variety cut out degree $n$ extensions, even when the variety is a complete intersection.  That is, if integers $m_\mathbf{a}$ are chosen for each $\mathbf{a} \in A$, as is done in the proof of Theorem \ref{thm:intro-summary}, the points $\{ \mathbf{x} \in (\mathbb{A}^n)^r : \mathrm{Tr}_{n,\mathbf{a}}(\mathbf{x}) = m_\mathbf{a} \text{ for all $\mathbf{a} \in A$}\}$ need not define a degree $n$ field extension of the rationals.  For example, let $n=5$ and $r=2$, and consider the set $A$ produced by Lemma \ref{lem:r=2}.  Choosing integers $m_\mathbf{a} \in [-10^{|\mathbf{a}|},10^{|\mathbf{a}|}]$ randomly in \verb^Magma^ and performing a Groebner basis computation, the authors find that a ``typical'' choice of $\{m_\mathbf{a} : \mathbf{a} \in A\}$ gives rise to solutions that define a degree $30$ extension of the rationals that may be realized as a degree $5$ extension of a sextic field $F/\mathbb{Q}$.  Thus, integers corresponding to the mixed traces of elements in quintic fields should be such that the polynomial typically defining this sextic field admits a rational root.  More generally, the authors speculate that mixed traces of $r$-tuples of integers in degree $n$ extensions lie on \emph{thin} subsets of $\mathbb{Z}^{rn}$ in the sense of Serre, and that an understanding of these subsets could yield a substantial improvement to the resulting bounds on number fields.

\section*{Acknowledgments}
Although our work is mostly independent of Couveignes's, we learned of Theorem \ref{thm:AH} from his paper, allowing us to streamline our proof and to improve
the constant in our exponents for $N_n(X)$.

We would like to thank Manjul Bhargava and Akshay Venkatesh for helpful feedback. 

FT was partially supported by grants from the Simons Foundation (Nos. 563234 and 586594), and RJLO was partially supported by NSF grant DMS-1601398.

\bibliographystyle{alpha}
\bibliography{references}
\end{document}